\documentclass[11pt,letterpaper]{amsart}
\usepackage[top=1.15in,bottom=1.15in,left=1.15in,right=1.15in,marginpar=1in]{geometry}
\usepackage{subcaption}
\usepackage{graphicx}
\usepackage{amssymb}
\usepackage{epstopdf}
\usepackage[mathcal]{eucal}
\usepackage{mathtools}  
\usepackage{pinlabel}
\usepackage{hyperref}
\usepackage{enumitem}
\usepackage{changepage}
\usepackage[all,cmtip]{xy}
\usepackage[textsize=tiny]{todonotes}
\usepackage{tikz-cd}
\usetikzlibrary{decorations.pathreplacing}
\usetikzlibrary{shapes.geometric, arrows}
\usepackage{scalerel}

\makeatletter
\newcommand*{\bigcdot}{}
\DeclareRobustCommand*{\bigcdot}{%
  \mathbin{\mathpalette\bigcdot@{}}%
}
\newcommand*{\bigcdot@scalefactor}{.75}
\newcommand*{\bigcdot@widthfactor}{1.15}
\newcommand*{\bigcdot@}[2]{%
  \sbox0{$#1\vcenter{}$}
  \sbox2{$#1\cdot\m@th$}%
  \hbox to \bigcdot@widthfactor\wd2{%
    \hfil
    \raise\ht0\hbox{%
      \scalebox{\bigcdot@scalefactor}{%
        \lower\ht0\hbox{$#1\bullet\m@th$}%
      }%
    }%
    \hfil
  }%
}
\makeatother
\newcommand{\nc}{\newcommand}
\nc{\dmo}{\DeclareMathOperator}
\nc{\nt}{\newtheorem}

\nt{theorem}{Theorem}[section]
\nt{proposition}[theorem]{Proposition}
\nt{corollary}[theorem]{Corollary}
\nt{move}{Move}
\nt{lemma}[theorem]{Lemma}
\nt*{tpdp}{Topological polynomial decision problem}

\theoremstyle{definition}
\nt{function}{Function}

\newcounter{step}
\newenvironment{step}[1][]{{\refstepcounter{step}\par
   \noindent \textbf{Step~\thestep : #1.} \rmfamily}\\}{\bigskip} 
   \newcommand{\steplabel}[1]{\addtocounter{step}{-1}\refstepcounter{step}\label{#1}}
   \newcounter{substep}
\counterwithin{substep}{step}

\let\ORIincludegraphics\includegraphics
\renewcommand{\includegraphics}[2][]{\ORIincludegraphics[scale=0.91666,#1]{#2}}

\dmo{\Mod}{Mod}
\dmo{\Teich}{Teich}
\dmo{\PMod}{PMod}
\dmo{\LMod}{LMod}
\dmo{\Homeo}{Homeo}
\dmo{\Aut}{Aut}
\dmo{\Sq}{Sq}
\dmo{\Cub}{Cub}
\dmo{\Rot}{Rot}

\nc{\C}{\mathbb C}
\nc{\Z}{\mathbb Z}
\nc{\Q}{\mathbb Q}
\nc{\N}{\mathcal N}
\nc{\M}{\mathcal M}
\nc{\R}{\mathbb R}
\nc{\A}{\mathcal A}
\nc{\Y}{\mathcal Y}
\nc{\T}{\mathcal T}

\definecolor{cerulean}{RGB}{0, 188, 255}
\definecolor{dustypink}{RGB}{224, 157, 205}

\nc{\p}[1]{\medskip\noindent\emph{#1.}}

\title[Transitivity of Pure Hurwitz]{Transitivity of the pure Hurwitz classes of quadratic post-critically finite polynomials}
\author{Yvon Verberne, Rebecca R. Winarski }
\date{February 2023}

\begin{document}

\maketitle
\begin{abstract}
    We prove that for two post-critically finite quadratic polynomials $g,h$, there is a mapping class $\phi$ of the sphere with finitely many marked points such that $g\phi$ and $h$ are pure Hurwitz equivalent.
\end{abstract}
\section{Introduction}
The study of the dynamics of post-critically finite rational maps dates back to the work of Fatou and Julia in the 1910s and has gained traction since the work of Thurston in the 1980s.  Moreover, the pure mapping class group acts on the set of post-critically finite branched covers with a given post-critical set and preservers the branching data of the maps.  This action has been of considerable interest by Bartholdi--Dudko, Bartholdi--Nekrashevych, Pilgrim, Kelsey--Lodge, D. Thurston, Koch--Pilgrim--Selinger (see \cite{BD,BD2, BaNe, FPP2,FPP,KL,KPS,lodge,pilgrim_alg}, for example), and connects holomorphic dynamics, geometric group theory and surface topology.  On the other hand, impure mapping class groups also act on the set of post-critically finite branched covers, but may change the combinatorics of the post-critical set of the branched cover.  This action is a priori richer, but requires an explicit understanding of how impure mapping classes change the combinatorics of branching data.  In this paper, we make a first step in this direction, showing that for any two pure Hurwitz classes -- a combinatorial description of the dynamics of the branched cover -- there is an impure mapping class between representatives of the classes in the quadratic case.

Let $f:S^2\rightarrow S^2$ be a branched cover.  Let $C\subset S^2$ be the set of ramified points of $f$, that is, the set of points for which $f$ is not locally injective.  We will call the set of points $C$ the {\it critical points} of $f$, matching the terminology for holomorphic maps.  The {\it post-critical set} of $f$ is the set $P=\displaystyle\bigcup_{n\geq 1}f^n(C)$.  The map $f$ is said to be {\it post-critically finite} if $|P|<\infty$.  In this paper, we focus on marked branched covers, that is, a pair consisting of a branched cover $f:S^2\rightarrow S^2$ and a (finite) set of points $M$, such that $f(M)\subset M$.

Let $f$ be a marked post-critically finite branched cover with marked set $M_f$ such that $M_f$ contains the critical points and the post-critical set of $f$.  There is a finite directed graph that describes the orbit of each of the critical points of $f$.  A {\it portrait} or dynamical portrait for $f$ is a labeled directed graph in which the vertices correspond to points in $M_f$. The vertices corresponding to critical points will be denoted by $\ast$. Let $p_1,p_2$ be elements of $M$.  Let $v_i$ be the vertex of the portrait of $f$ corresponding to $p_i$.  There is an edge of the portrait from $v_1$ to $v_2$ if $f(p_1)=p_2$.  The edge will be labeled with the local degree of $f$ at $p_1$.

\p{Combinatorial and portrait equivalence} In the study of holomorphic dynamics, post-critically finite marked branched covers are often considered up to combinatorial (or Thurston) equivalence.  That is: let $g$ and $h$ be post-critically finite marked branched covers with marked sets $M_g$ and $M_h$, respectively.  We say that $g$ and $h$ are {\bf combinatorially equivalent} if there exist orientation preserving homeomorphisms $\psi_1,\psi_2$ such that:
\begin{enumerate}
\item $\psi_1g=h\psi_2$,
\item $\psi_i(M_g)=M_h$ for $i\in\{1,2\}$, and 
\item $\psi_1$ and $\psi_2$ are homotopic relative to $M_g$.
\end{enumerate}
The branched covers $g,h$ are {\it pure Hurwitz equivalent} if the first two conditions above are satified and $\psi_0(p)=\psi_1(p)$ for all $p\in P_g$.  Kelsey--Lodge prove that two post-critically finite quadratic rational maps are pure Hurwitz equivalent if and only if their dynamical portraits are isomorphic as directed graphs \cite[Lemma 3.6]{KL}.  Our main theorem is that given any two pure Hurwitz classes, there is a mapping class that takes a representative of one to a representative of the other.

\begin{theorem}\label{thm:main}
Let $g,h:S^2\rightarrow S^2$ be marked quadratic branched covers with dynamical portraits $\Gamma_g$ and $\Gamma_h$.  There exists a finite set of points $M\subset S^2$ that contains the post-critical set for $g$ and $h$ and $\phi\in\Mod(S^2,M)$ so that $g\circ \phi$ is pure Hurwitz equivalent to $h$.
\end{theorem}

In particular, $g$ and $h$ may have different numbers of post-critical points.  Using an impure mapping class $\phi$ allows us to increase or decrease the number of post-critical points of the branched cover. 

We contrast our main result with that of Kelsey--Lodge \cite{KL} who find M\"obius transformations between {\it combinatorial classes} of post-critically finite branched covers with four or fewer post-critical points (that are not nearly Euclidean Thurston maps \cite{FPP2,FPP}). 

\p{Strategy} We prove Theorem \ref{thm:main} by developing an algorithm that finds a sequence of marked branched covers $\{(g,M_g)=(f_0,M_0),(f_1,M_1),\cdots,(f_n,M_n)=(h,M_h)\}$ such that each $f_i$ differs from $f_{i+1}$ by what we call a (portrait) function.  The functions are half twists either between points in $M_i$ or between one point in $M_i$ and a preimage of a critical point.  If the function by which we compose is a half twist between a point in $M_i$ and another point $\diamond$, we take $M_{i+1}=M_i\cup\{\diamond\}$.

The overall strategy is as follows: start with the marked branched covers $(g,M_g)=(f_0,M_0)$ and $(h,M_h)$.  We create sequences of marked branched covers $\{(f_{i},M_i)\}$ where the portrait for $f_{i+1}$ differs from the portrait from $f_i$ by one of the functions \ref{Function:Split}-\ref{Function:DecreasePeriod} defined in Section \ref{sec:functions}. 
 This process terminates when we reach a branched cover $f_n$ that has the same portrait as $z^2$.  We repeat the process for $(h,M_h)$.  Then invert the sequence from $(h,M_h)$ to $z^2$ to create a sequence from $(g,\Gamma_g)$ to $(h,\Gamma_h)$.

 \p{Marked set} Theorem \ref{thm:main} states that marked branched covers $(g,M_g)$ and $(h,M_h)$ that for {\it some} marked set $M$, there is a mapping class $\phi\in\Mod(S^2,M)$ for which the portrait of $g\phi$ is isomorphic to the portrait of $h$.  As indicated above, the algorithm determines this set by first beginning with the post-critical set of $M_0=M_f$ and then adding additional points at some stages of the algorithm.  At the $i$th stage, we add at most one point to the set $M_{i-1}$. 
The point that is (possibly) added is a preimage of a point that is a critical point of $f_i$.  We repeat the algorithm with $h$ to obtain another set of points consisting of $M_h$ and a union of points obtained in each step.  Because the algorithm terminates in a finite number of steps, the set $M$ consists of the union of $M_g\cup M_h$ with a finite set of points.

\subsection{Future directions} 
Much of the study of actions of mapping class groups on sets of post-critically finite branched covers has restricted to actions of {\it pure} mapping class groups.  The primary goal of this paper is to promote the study of the action of {\it impure} mapping class groups on post-critically finite branched covers. 
Moreover, the functions we present in Section \ref{sec:functions} to change dynamical portraits provide a tool to use arguments that induct on the number of points of a post-critical set in families of post-critically finite branched covers.

\p{Braids versus symmetric groups} As indicated above, we construct a sequence of marked branched covers $\{(g,M_g)=(f_0,M_0),(f_1,M_1),\cdots,(f_n,M_n)=(h,M_h)\}$, and keep track of a sequence of mapping classes $\{\phi_1,\cdots, \phi_n\}$ for which $f_{i+1}=f_i\phi_{i+1}$.  The functions are agnostic to the choice of mapping class up to its action on the marked set $M_i$.  We show that we can ``move between" different pure Hurwitz equivalence classes by composing by an impure mapping class.  With additional information about the actual mapping class (rather than just its quotient to a finite symmetric group), it is possible that one can make stronger statements about finer equivalence classes (such as combinatorial equivalence).

\p{Twisting problems} Thurston invigorated the study of post-critically finite branched covers $S^2\rightarrow S^2$ by proving that outside of a family of well-understood examples, every post-critically finite branched cover $S^2\rightarrow S^2$ is either isotopic to the conjugate of a rational map or has a certain type of topological obstruction.  Thurston's work has inspired decades of work on post-critically finite branched covers $S^2\rightarrow S^2$, including twisting problems, which we describe as follows.  Let $f$ be a rational map with post-critical set $P_f$.  The pure mapping class group $\PMod(S^2,P_f)$ is the group of homotopy classes of orientation preserving homeomorphisms of $S^2$ that fix the set $P_f$.  A twisting problem is the following: let $\Gamma$ be a subgroup of $\PMod(S^2,P_f)$.  For all $\phi\in\Gamma$ determine whether $\phi\circ f$ is obstructed or which rational map it is equivalent to.  The classic such example of a twisting problem is the {\it twisted rabbit problem}, presented by Hubbard and solved by Bartholdi--Nekrashevych \cite{BaNe}.

\p{Realiziability of rational maps} Our algorithm gives no information about whether a portrait can be realized as a {\it rational} map, as in Floyd--Kim--Koch--Parry--Saenz \cite{FKKPS}.  Theorem \ref{thm:main} gives us that for any pair of portraits, there is always an impure mapping class by which we can compose some representative with one portrait to obtain a representative with the other portrait.  Therefore, the realizability of a portait (as described by Floyd--Kim--Koch--Pary--Saenz) is not invariant under the action of an impure mapping class group.  A natural question is with finer details about a combinatorial class of a rational map and the exact mapping class used, whether it is possible to construct paths of rational maps within the space of post-critically finite branched covers $S^2\rightarrow S^2$?

\p{Generalization to higher degree} We expect that Theorem \ref{thm:main} should hold if we replace quadratic branched covers $S^2\rightarrow S^2$ with any Hurwitz class of branched covers $S^2\rightarrow S^2$, that is: for post-critically finite branched covers of the sphere of any degree where the critical values have the same ramification data.  Our algorithm does use the hypothesis that the branched covers under consideration are quadratic in significant ways.  In particular, we assume that the branched covers under consideration have exactly two critical points of ramification two and we use the fact that every post-critical point that is not a critical value has exactly two preimages.  However, we conjecture that with the appropriate modifications for any Hurwitz class, one could obtain an analogous result to Theorem \ref{thm:main}.

\subsection{Outline of paper} In Section \ref{sec:portaits_feature} we establish that a dynamical portrait of a quadratic marked branched cover can be described by the number of components of the portrait, the number and structure of the pre-periods of each component, the pre-period length(s), and the cycle length(s).  In Section \ref{sec:functions} we specify the functions that modify the dynamical portraits in the desired way.  In Section \ref{sec:algorithm} we give the algorithm that finds the mapping class $\phi$ in Theorem \ref{thm:main}.  Finally, in Section \ref{sec:proof}, we prove that the algorithm terminates in finitely many steps, thus completing a proof of Theorem \ref{thm:main}.

\subsection{Acknowledgments} The authors would like to thank ICERM where the project originated.  The authors would like to thank Russell Lodge, Dan Margalit, Kevin Pilgrim, and Kevin Walsh for conversations about the paper. The authors would like to thank Kevin Pilgrim and the University of Indiana for hosting the authors at the Early-Career Complex Dynamics workshop in summer 2024 and 2025 and Giulio Tiozzo for inviting the second author to University of Toronto.  The first author was partially supported by an NSERC PDF, and was partially supported by NSERC Discovery Grant RGPIN-2024-05587.  The second author was supported by the National Science Foundation under Grant No.\ DMS-2002951.

\section{Portrait features}\label{sec:portaits_feature}  
In this section, we define dynamical portraits and describe the combinatorics of the portraits, which we call {\it portrait features}.  In particular, we define enough portrait features to describe a portrait up to isomorphism, which is the content of Proposition \ref{prop:same}.

\subsection{Abstract portraits}
We say that a labeled directed graph $\Gamma$ is an {\it abstract portrait} if:
\begin{enumerate}
\item All vertices of $\Gamma$ belong in at least one of two sets: a set of points distinguished as critical points $C$ and a set of points distinguished as post-critical points $P$.  The intersection $P\cap C$ need not be empty.
\item Every vertex of $\Gamma$ is the origin of exactly one directed edge that terminates at another vertex of $\Gamma$.
\item Each directed edge is labeled with a positive integer, and a directed edge that originates at a point in $C$ is labeled with an integer greater than 1.
\item The set $P$ exactly consists of the vertices that are the terminus of at least one directed edge. 
\item Every component of $\Gamma$ contains at least one point in $C$.
\end{enumerate}
We will use several consequences of this definition:
\begin{enumerate}[label=(\Alph*)]
    \item every component of $\Gamma$ contains exactly one cycle,
    \item two edges originating at different periodic points cannot have the same terminus,
            \item for every vertex $p\in P$, there exists a directed path from some $c\in C$ to $p$, and
    \item for every vertex $v$ of $\Gamma$ that does not belong to a cycle of $\Gamma$, there is a directed path beginning at $v$ and ending at a vertex in the cycle of $\Gamma$.

\end{enumerate}
Consequences (A) and (B) follow from (2) above.  Consequence (C) follows inductively from (4) and (5).  Consequence (D) follows from consequence (C) and (2).

We say that any vertex in a cycle of $\Gamma$ is {\it periodic} in $\Gamma$ and any vertex that is not in a cycle of $\Gamma$ is {\it (strictly) pre-periodic}.  The set of pre-periodic points in $P$ is called a {\it pre-period} of $\Gamma$.  Because $\Gamma$ is directed, each pre-period has a first point that is point in $C$.

\p{Abstract portraits for marked branched covers} Let $f$ be a marked post-critically finite branched cover.  There is an abstract portrait associated to $f$ in which the vertices correspond to a subset of the marked set $M_f$. The vertices corresponding to critical points are labeled $\ast$.  Let $p_1$ and $p_2$ be points in $M_f$, and let $v_i$ be the vertex of the portrait of $f$ corresponding to $p_i$.  There is an edge of the portrait from $v_1$ to $v_2$ if $f(p_1)=p_2$.  The edge will be labeled with the local degree of $f$ at $p_1$.  The label will be greater than one if and only if $v_1$ is a critical point of $f$.  We will suppress the labels when their value is one, or if they are not relevant to the discussion.  We observe that because $f$ is a branched cover, the sum of the labels of the arrows that terminate at a vertex cannot exceed the degree of the branched cover.  We can think of $f$ as acting on its dynamical portrait by mapping a vertex $v$ of the portrait of $f$ to the vertex at the terminus of the edge originating at $v$.  In this way, we will think of the vertices of a portrait as the orbit of $f$.

This paper focuses on quadratic branched covers, which have additional restrictions.  We say that an abstract portrait $\Gamma$ is {\it compatible with a quadratic branched cover} if:
\begin{itemize}
    \item the set of vertices of $\Gamma$ distinguished as critical points $C$ has exactly two points,
    \item each edge originating at a point of $C$ is labeled with a ``2",
    \item the sum of the labels of edges that terminate at each vertex is at most 2.
\end{itemize} 
Because much of the focus of the paper is on {\it marked} branched covers, we additionally require that the marked set of a marked branched cover must contain the set of vertices of the portrait, but may (and likely will) contain additional marked points that are not part of the abstract portrait.  Therefore everything that follows will also apply to marked quadratic branched covers where the vertex set is contained in the marked set.

\subsection{Portrait features} An abstract portrait that is compatible with a quadratic branched cover can be described in terms of its features.  

Let $\Gamma$ be an abstract portrait that is compatible with a marked quadratic branched cover. Let $C=\{C_1,C_2\}$ be the set of points distinguished as critical points.

\p{Components} Every component of $\Gamma$ contains at least one point that is distinguished as a critical point. Because $\Gamma$ is compatible with a quadratic branched cover, it has only two points distinguished as a critical point.  Therefore $\Gamma$ has either one or two connected components.
  
\p{Number of pre-periods} As above, any point contained in the cycle of a graph is {\it periodic}.  If a point distinguished as a critical point is not periodic, there must be a directed path from it to (at least one) periodic point and we say that it is {\it (strictly) pre-periodic}.  If a point distinguished as a critical point $C_j$ for $j\in\{1,2\}$ is pre-periodic, there is a (unique) shortest path from $C_j$ to a cycle of the graph $\{C_j,P_1,\cdots,P_r\}$ where $P_r$ is periodic and no other point in the path is periodic.  We call the set of points $\{C_j,P_1,\cdots,P_{r-1}\}$ the {\it pre-period of $C_j$}. 
Because every pre-period must begin with a point distinguished as a critical point, $\Gamma$ can have either zero, one, or two pre-periods (at most one for each point distinguished as critical points that it contains). If $\Gamma$ has two components, each component contains one point distinguished as a critical point.  Each component has one pre-period for each pre-periodic critical point, so each component can have zero, one, or two pre-periods.

\p{Pre-period patterns} If $\Gamma$ has one component and both points of $C$ are strictly pre-periodic, their pre-periods can either be disjoint or can intersect non-trivially.   
If the pre-periods of $C_1$ and $C_2$ intersect non-trivially, there are two additional possibilities:\begin{itemize}
         \item either the pre-period of $C_1$ contains a post-critical point that is not in the pre-period of $C_2$ and the pre-period of $C_2$ contains a post-critical point that is not in the pre-period of $C_1$ or
            \item one point distinguished as a critical point is in the pre-period of the other.
     \end{itemize} 
We will refer to the pre-periods in the former case as {\it distinct and intersecting.}  If $C_2$ is not in the pre-period of $C_1$, but the pre-periods of $C_1$ and $C_2$ intersect non-trivially, then the pre-period of $C_2$ must contain a post-critical point that is not in the pre-period of $C_2$.  Otherwise the portrait is not compatible with a quadratic branched cover.  If one point distinguished as a critical point is in the pre-period of the other, then one pre-period is {\it contained in the other}.  In both cases, there is a unique first point of the intersection of the pre-periods and a unique first periodic point.

\p{Length of pre-period and cycle} As above, every component of a portrait has exactly one cycle (possibly just a single point with an edge to itself).  We call the number of vertices in the cycle the {\it length of the cycle} of the component.  If $C_j$ is the only point of $C$ in a component, we can refer to the {\it cycle of $C_j$} as the cycle in the component containing $C_j$.  We denote the length of the cycle of $C_j$ by $k_j$.  If $C_j$ for $j\in\{1,2\}$ is strictly pre-periodic, we call the number of points in its pre-period the {\it length of the pre-period of $C_j$}.  As above, we will denote the length of the pre-period of $C_j$ as $r_j$.  Moreover, the vertices of the pre-period are ordered as the shortest directed path from $C_j$ to a periodic point (the periodic point is not part of the pre-period).  If a component of $\Gamma$ contains both of the vertices in $C$, we record other distances, dependent on the pre-period structure, which we explain as follows.

If a component of $\Gamma$ contains both elements of $C$, both points of $C$ are strictly pre-periodic, and one vertex in $C$ is contained in the pre-period of the other, we label the points of $C$ so that the pre-period of $C_2$ is contained in the pre-period of $C_1$.  Then there exists a directed path from $C_1$ to $C_2$ (and no directed path from $C_2$ to $C_1$).  Let $r_j$ be the length of the pre-period of $C_j$.  The length of the directed path from $C_1$ to $C_2$ will have length $r_1-r_2$. 

If a component of $\Gamma$ has two pre-periods that are distinct and intersecting, we can specify the length of the cycle ($k$, as above), the length of the intersection of the pre-period of $C_1$ and $C_2$, the length of the pre-period of $C_1$ that is unique from the pre-period of $C_2$, and the length of the pre-period of $C_2$ that is unique from the pre-period of $C_1$.

If a component of $\Gamma$ contains elements of $C$ and has no pre-periods, then both $C_1$ and $C_2$ are periodic.  Then there exists a directed path from $C_1$ to $C_2$.  The length of the (unique) shortest such directed path is called the {\it distance between $C_1$ and $C_2$}, which we denote by $k_1$ as above.  Likewise, there is a (unique) shortest directed path from $C_2$ to $C_1$, and its length is denoted by $k_2$. 
 We call this the {\it distance between $C_2$ and $C_1$}.

If a component of $\Gamma$ contains both vertices in $C$ and one is periodic and the other is pre-periodic, then $\Gamma$ has one pre-period.  Label the periodic point of $C$ as $C_2$ and the strictly pre-periodic point of $C$ as $C_1$. 
 In particular, there will be a directed path from $C_1$ to $C_2$ but not from $C_2$ to $C_1$.   Let $q$ be the length of the (shortest) directed path from $C_1$ to $C_2$ and let $r$ be the length of the pre-period of $C_1$ as above.  We refer to $q-r$ as the {\it position of $C_2$ in the cycle of $C_1$.}  We can also measure the length of the shortest path from $C_2$ to the first periodic point, which will be $k-(q-r)$.

 If a component of $\Gamma$ has two disjoint pre-periods, each vertex in $C$ has directed path to the set of periodic points.  In particular, each point of $C$ has a shortest path to a periodic point.  Let $T_j$ be the first periodic point any directed path originating at $C_j$ contains (ie. the terminal vertex of the shortest path from $C_j$ to a periodic vertex).  As above, we will let $r_j$ be the length of the directed path from $C_j$ to $T_j$ for $j\in\{1,2\}$.  Because both $T_1$ and $T_2$ are periodic, there is a directed path from $T_1$ to $T_2$ and a directed path from $T_2$ to $T_1$.  We call length of the shortest path from $T_1$ to $T_2$ {\it distance between the pre-period of $C_1$ and the pre-period of $C_2$} and denote it by $k_1$ (since it is a distance in the cycle).  Likewise, we call the length of the shortest path from $T_2$ to $T_1$ {\it distance between the pre-period of $C_2$ and the pre-period of $C_1$} and denote it by $k_2$.  That is, we track four lengths: $r_1$, the length of the pre-period of $C_1$, $r_2$ the length of the pre-period of $C_2$, $k_1$ the distance between the pre-period of $C_1$ and $C_2$, and $k_2$ the distance between the pre-period of $T_2$ and $T_1$. The quantities $r_1,r_2,k_1$ and $k_2$ are indicated in Figure \ref{fig:disjoint_distances}.

When the portrait has one component, there are five possible pre-period structures.  If both critical points are strictly pre-periodic, and their pre-periods intersect non-trivially, we need only measure the length of the cycle.  In the remaining three cases, we observe that there are two points in the cycle of $\Gamma$ between which we measure distance: 
\begin{enumerate}
    \item when there are no pre-periods, we measure the distance between the critical points in the cycle,
    \item when there is one pre-period, we measure the position of $C_2$ in the cycle of $C_1$: that is, the distance between the pre-period and $C_2$, and
    \item when there are two disjoint pre-periods, we measure the distance between $T_1$ and $T_2$.
\end{enumerate}
To unify these three cases, let $R_j$ be the first point in a directed path beginning at $C_j$ (including $C_j$) that is periodic.  Let $k_1$ be the length of the shortest path between $R_1$ and $R_2$ and let $k_2$ be the length of the shortest path between $R_2$ and $R_1$.  In these cases, we call the lengths $k_1$ and $k_2$ {\it cycle lengths}.  If $k_1=0$ then $C_1$ must be pre-periodic and $C_2$ must be periodic and the first periodic point in the cycle (of $C_1$).  Our assumption that if one critical point is pre-periodic and the other is periodic is that $C_1$ is the pre-periodic critical point. This excludes the possibility that $k_2=0$. 
 In the cases where the pre-periods of $C_1$ and $C_2$ intersect non-trivially, there is only one cycle length, $k$.

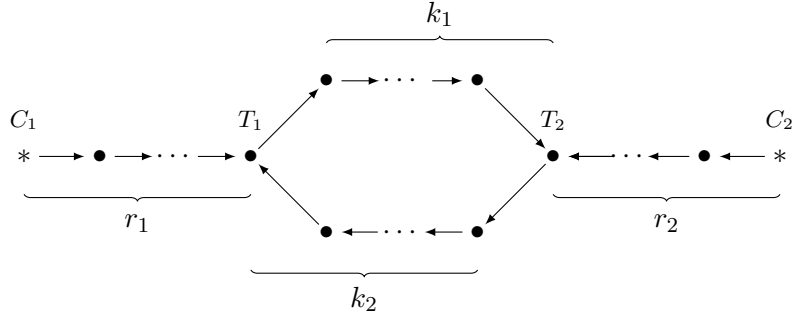
\begin{figure}
\begin{tikzpicture}\label{fig:cycle_lengths}
\node[label={\footnotesize $C_1$}] at (3,2) {$\ast$};
\node at (4,2) {$\bullet$};
\node at (5,2) {$\cdots$};
\node[label={\footnotesize $T_1$}] at (6,2) {$\bullet$};
\node at (7,3) {$\bullet$};
\node at (8,3) {$\cdots$};
\node at (9,3) {$\bullet$};
\node[label={\footnotesize $T_2$}] at (10,2) {$\bullet$};
\node at (11,2) {$\cdots$};
\node at (12,2) {$\bullet$};
\node[label={\footnotesize $C_2$}] at (13,2) {$\ast$};
\node at (7,1) {$\bullet$};
\node at (8,1) {$\cdots$};
\node at (9,1) {$\bullet$};
 \draw [-latex] (3.2,2) -- (3.8,2);
  \draw [-latex] (4.2,2) -- (4.7,2);
   \draw [-latex] (5.3,2) -- (5.8,2);
  \draw [-latex] (11.8,2) -- (11.25,2);
  \draw [-latex] (12.8,2) -- (12.2,2);
\draw [-latex] (10.75,2) -- (10.2,2);
\draw [-latex] (6.1,2.1) -- (6.9,2.9);
\draw [-latex] (10.8,2) -- (10.2,2);
  \draw [-latex] (8.4,3) -- (8.8,3);
  \draw [-latex] (9.1,2.9) -- (9.9,2.1);
    \draw [-latex] (9.9,1.9) -- (9.1,1.1);
   \draw [-latex]  (7.2,3) -- (7.7,3);
      \draw [-latex]  (7.68,1) -- (7.2,1);
     \draw [-latex]  (8.8,1) -- (8.3,1);
     \draw [latex-]  (6.1,1.9) -- (6.9,1.1);
\draw [decorate,
    decoration = {brace,mirror}] (3,1.5) --  (6,1.5) node[midway,yshift=-1em]{$r_1$};
    \draw [decorate,
    decoration = {brace,mirror}] (6,.5) --  (9,.5) node[midway,yshift=-1em]{$k_2$};
    \draw [decorate,decoration = {brace,mirror}] (10,1.5) --  (13,1.5) node[midway,yshift=-1em]{$r_2$};
    \draw [decorate,
    decoration = {brace}] (7,3.5) --  (10,3.5) node[midway,yshift=1em]{$k_1$};
     
\end{tikzpicture}
\caption{When a portrait has two disjoint pre-periods, we measure the length of each pre-period $r_1-1$ and $r_2-1$ and the distances between the pre-periods, $k_1$ and $k_2$.}\label{fig:disjoint_distances}
\end{figure}

The number of components, number of pre-periods, pre-period patterns, lengths of pre-periods (and their intersection, if relevant), and distance between critical points or pre-periods are sufficient information to determine the portrait of a quadratic branched cover of $S^2\rightarrow S^2$.

\begin{proposition}\label{prop:same}
    Let $\Gamma^1$ and $\Gamma^2$ be two abstract portraits that are compatible with quadratic maps.  If $\Gamma^1$ and $\Gamma^2$ have the same number of components, the same number of pre-periods, the same pre-period patterns, the same pre-periods lengths, and the same cycle lengths, then $\Gamma^1$ and $\Gamma^2$ are isomorphic as abstract (directed) graphs.
\end{proposition}
In the proof below, we define an isomorphism by sending one critical point of $\Gamma^1$ to a critical point of $\Gamma^2$.  However, we have to choose which critical point of $\Gamma^1$ is sent to each critical point of $\Gamma^2$, and the choice depends on the portrait features.  Hence, the proof below is entirely laying out casework to explicitly explain which critical point of $\Gamma^1$ maps to each critical point of $\Gamma^2$.
\begin{proof}
We will define a map $\eta:\Gamma^1\rightarrow\Gamma^2$ by specifying how to map the points distinguished as critical points.  In what follows we will refer to these points as ``critical points" for simplicity, with the understanding that there is not necessarily an associated map for which they actually are critical points.  Below we specify how to label the critical points in $\Gamma^1$ and $\Gamma^2$ to ensure $\eta$ is an isomorphism.

If $\Gamma^1$ (and hence $\Gamma^2$) has two (non-empty) components, then each component will contain one critical point.  If the components are isomorphic, the choice of labeling of critical points does not matter, so label the critical points of $\Gamma^1$ by $C_1^1$ and $C_2^1$ and the critical points of $\Gamma^2$ by $C_1^2$ and $C_2^2$.  If one of the components is a cycle and the other is not, let $C_2^1$ be the critical point in the cyclic component of $\Gamma^1$ and let $C_2^2$ be the critical point in the cyclic component of $\Gamma^2$.  Likewise, label the critical point in the pre-periodic component of $\Gamma^i$ by $C_1^i$ for $i\in\{1,2\}$. 

If both components are cyclic, let $C_1^1$ be the critical point in the longer cycle of $\Gamma^1$ and let $C_1^2$ be the critical point in the longer cycle of $\Gamma^2$.  Likewise, let $C_2^i$ be the critical point in the shorter cycle of $\Gamma^i$ for $i=\{1,2\}$.

If both components have a pre-period and are not isomorphic, either the pre-periods or the cycles have to be different lengths (or both).  If the pre-periods are of different length, let $C_1^i$ be the critical point in $\Gamma^i$ with the longer pre-period and let $C_2^i$ be the critical point in $\Gamma^i$ with the shorter pre-period.  If the pre-periods are the same length but the cycles are of different length, let $C_1^i$ be the critical point in $\Gamma^i$ with the longer cycle and let $C_2^i$ be the critical point in $\Gamma^i$ with the shorter cycle.

If $\Gamma^1$ (and hence $\Gamma^2$) has one component, then both critical points are in that component.  If the component of $\Gamma_1$ is has no pre-periods, then it is cyclic.  In this case, there is a (shortest) directed path from each critical point to the other.  If the lengths of the paths are equal, the choice of labeling of critical points does not matter, so label the critical points of $\Gamma^1$ by $C_1^1$ and $C_2^1$ and the critical points of $\Gamma^2$ by $C_1^2$ and $C_2^2$.  If the lengths of the directed paths are unequal, label $C_1^i$ as the critical point of $\Gamma^i$ at the origin of the longer path and $C_2^i$ as the critical point in $\Gamma^i$ at the origin of the shorter path.  

If the graph $\Gamma^1$ (and hence $\Gamma^2$) has one component and one pre-period, for $i\in\{1,2\}$ let $C_1^i$ be the critical point in $\Gamma^i$ that is strictly pre-periodic.  The other critical point in $\Gamma^i$ is thus periodic and hence is in the cycle of $C_1^i$.  Label the other critical point $C_2^i$.

If both critical points in $\Gamma^1$ (and hence also in $\Gamma^2$) are pre-periodic, then either the pre-periods are of equal length or they are not.  If one pre-period is longer, let $C_1^i$ be the critical point in $\Gamma^i$ with the longer pre-period and let $C_2^i$ be the critical point in $\Gamma^i$ with the shorter pre-period.  If the pre-periods are of equal length, there are two cases:
\begin{enumerate}
\item If the pre-periods of the critical points are the same length in $\Gamma^1$ (and hence also in $\Gamma^2$) and distinct and intersecting, the choice of labeling of critical points does not matter, so label the critical points of $\Gamma^1$ by $C_1^1$ and $C_2^1$ and the critical points of $\Gamma^2$ by $C_1^2$ and $C_2^2$.

\item If both critical points in $\Gamma^1$ have disjoint pre-periods of the same length, there is a unique first periodic point at paths originating at each of the critical points.  Let $k_1^i$ and $k_2^i$ be the length of the two directed path from one pre-periodic point to the other, such that $k_1^i\geq k_2^i$ for each $i\in\{1,2\}$.  If $k_1^1=k_2^1$ (and by assumption $k_1^2=k_2^2$), it does not matter how we label the critical points, so label the critical points of $\Gamma^i$ by $C_1^i$ and $C_2^i$ for each $i\in\{1,2\}$.  If $k_1^1>k_2^1$ (and by assumption $k_1^2>k_2^2$), let for each $i,j\in\{1,2\}$, $T_j^i$ be the origin of the path of length $k_j^i$.  Then for $i,j\in\{1,2\}$, let $C_j^i$ be the critical point for which $T_j^i$ is the first periodic point at a path originating at $C_j^i$.
\end{enumerate}
    Let $\eta:\Gamma^1\rightarrow \Gamma^2$ be defined by $\eta(C_1^1)=C_1^2$ and $\eta(C_2^1)=C_2^2$.  For any other vertex $v^1$ in $\Gamma^1$, there is a directed path from either $C_1^1$ or $C_2^1$ (or possibly both) to $v^1$. Hence we inductively define $\eta$ as follows.  Let $v$ be a vertex of $\Gamma^1$ with $\eta(v)=w$.  Because $\Gamma^1$ is an abstract portrait, $v$ is the origin of exactly one edge in $\Gamma^1$.  Let $t_v$ be the terminus of the edge originating at $v$.  Likewise, $w$ is the origin of exactly on edge in $\Gamma^2$.  Let $t_w$ be the terminus of the edge originating at $w$ in $\Gamma^2$.  Define $\eta(t_v)=t_w$.

    We first show that $\eta$ is well-defined.  In particular, suppose $o_1$ and $o_2$ are vertices of $\Gamma^1$ such that $o_1$ and $o_2$ are both the origin of a directed edge to the same vertex $v$ of $\Gamma^1$.  We need to show that the definition of $\eta(v)$ is the same regardless of which edge (the edge originating at $o_1$ or the edge originating at $o_2$ we use to define $\eta(v)$.  We do so by showing that the edges that originate $\eta(o_1)$ and $\eta(o_2)$ have the same terminus.  By Consequence (B) above, there are two possibilities: either $o_1$ and $o_2$ are strictly pre-periodic or one of $\{o_1,o_2\}$ is pre-periodic and the other is periodic.

    If $o_1$ and $o_2$ are both strictly pre-periodic, they each belong to a pre-period of a different critical point.  Indeed, by definition a pre-period cannot have a cycle, so each element of a single pre-period of $C_1^1$ (respectively $C_2^1$) must map to a different vertex.  Therefore $o_1$ and $o_2$ must each belong to the pre-period of a different critical point.  Label them so that $o_i$ is in the pre-period of $C_i^1$.  Then $v$ is (the first point) in the intersection of the pre-periods of $C_1^1$ and $C_2^1$.  Because $\Gamma^1$ and $\Gamma^2$ have the same pre-period structure, the pre-periods of $C_1^2$ and $C_2^2$ must also be distinct and intersecting.  We claim that $\eta(v)$ is the first point of intersection of the pre-periods of $C_1^2$ and $C_2^2$.  Indeed, this follows from the assumption that the length of the pre-period of $C_1^1$ that is unique from the pre-period of $C_2^1$ is the same as the length of the pre-period of $C_1^2$ that is unique from the pre-period of $C_2^2$ and that the length of the pre-period of $C_2^1$ that is unique from the pre-period of $C_1^1$ is the same as the length of the pre-period of $C_2^2$ that is unique from the pre-period of $C_1^2$.

    If one of $o_1$ and $o_2$ is strictly pre-periodic and the other is pre-periodic, let $o_1$ be the vertex that is pre-periodic and let $o_2$ be the vertex that is periodic.  For the purpose of the proof, define $\eta(v)$ to be the terminus of the edge originating at $\eta(o_1)$, and we will show that this is the same vertex as the terminus of the edge originating at $\eta(o_2)$.  Because $o_1$ is pre-periodic, it is in the pre-period of a critical point, call it $c$ ($o_1$ could be in the pre-period of both critical points, in which case we can call either one $c$).  In fact, $o_1$ is the last pre-periodic point in the pre-period of $c$. The lengths of the pre-period(s) of $\Gamma^1$ containing $o_1$ and the pre-period(s) of $\Gamma^2$ containing $\eta(o_1)$ are the same.  Therefore $\eta(o_1)$ must also be the last pre-periodic point of the pre-period of $\eta(c)$ in $\Gamma^2$.  Then $\eta(v)$ must be the first periodic point of the component of $\Gamma^2$ that contains it, and all points in a directed path originating at $\eta(v)$ must also be periodic.  There is a directed path from $o_1$ to $o_2$ and, in particular, this directed path contains $v$.  Label the vertices of the directed path $o_1\rightarrow v\rightarrow w_1\rightarrow\cdots\rightarrow w_n\rightarrow o_2$.  Then consider the image of this directed path under $\eta$: $\eta(o_1)\rightarrow \eta(v)\rightarrow \eta(w_1)\rightarrow\cdots\rightarrow \eta(w_n)\rightarrow \eta(o_2)$, which is a directed path in $\Gamma^2$.  In particular, all points $\eta(w_1),\cdots,\eta(w_n),\eta(o_2)$ must be periodic.  Because the component of $\Gamma^1$ containing $o_1$ and the component of $\Gamma^2$ containing $\eta(o_1)$ have the same number of periodic points, $\eta(v),\eta(w_1),\cdots,\eta(w_n),\eta(o_2)$ are all of the periodic points of this component of $\Gamma^2$.  Then $\eta(o_2)$ must map to $\eta(v)$, and in fact, $\eta$ is well-defined.

It follows that $\eta$ is an isomorphism: $\eta$ is a bijection because we can construct an inverse $\eta^{-1}$ by switching the roles of $\Gamma^1$ and $\Gamma^2$ in the definition of $\eta$ (that is, $\eta^{-1}$ maps $C_1^2$ to $C_1^1$ and $C_2^2$ to $C_2^1$). Moreover, $\eta$ maps directed edge between vertices $v_1$ and $v_2$ in $\Gamma^1$ to a directed edge between $\eta(v_1)$ and $\eta(v_2)$ in $\Gamma^2$ by construction (as does $\eta^{-1}$).  Hence $\eta$ is an isomorphism.
\end{proof}

\section{Portrait functions}\label{sec:functions}
Throughout this section, let $f$ be a marked post-critically finite quadratic branched cover with portrait $\Gamma_f$ and marked set $M_f$.  We will denote the critical points of $f$ by $C_1$ and $C_2$.  As discussed above, one critical point may be in the $f$-forward orbit of the other.  If this is the case, we assume (up to relabeling) that $C_2=f^q(C_1)$ for some $q$.  Let $P_f$ be the post-critical set of $f$ and let $N_f$ be the union of $P_f$, $C_1\cup C_2$, and $f^{-1}(\{C_1,C_2\})$. 
 Let $\phi\in\Mod(S^2,N_f)$.  The map $f\phi$ will modify the portrait of $\Gamma_f$ and will have the property that $f(\phi(N_f))\subset N_f$.  We call this modification a {\it (portrait) function}.  The list below describes the portrait functions that we use by their desired effect on the portrait, though some of the functions require specific features of the starting portrait.
 
 We make a few observations about the sets $P_f$, $M_f$, and $N_f$:
 \begin{enumerate}
 \item the set $M_f$ is part of the definition of the marked branched covering map $f$, which is indicated as part of the domain of each of the functions below,
      \item we assume that $P_f\cup\{C_1,C_2\}$ is contained in $M_f$ -- when $f$ is the branched cover initiated by the algorithm in Section \ref{sec:algorithm}, $M_f$ will be constructed to contain $P_f\cup\{C_1,C_2\}$,
     \item $P_f\cup\{C_1,C_2\}\subset N_f$ by definition, and
     \item the set $N_f$ was defined only for the purpose of describing the mapping classes by which we compose -- aside from the fact that $P_f\cup\{C_1,C_2\}\subset M_f\cap N_f$, it need not have any relationship to $M_f$ and will not be used again.
 \end{enumerate}

 We define (portrait) functions that do each of the following to $\Gamma_f$.
\begin{enumerate}
    \item Split the portrait from one to two components
        \item Move $C_2$ from a cycle to a pre-period (if $C_2$ is in the $f$-orbit of $C_1$)
    \item Decrease a cycle length
\end{enumerate}
For each function, we specify by which element $\phi$ of $\Mod(S^2,N_f)$ we pre-compose $f$ by to modify the portrait as desired.  We then prove that the proposed composition modifies the portrait as desired.  In the algorithm below, we apply the functions  to target the portrait for $z^2$, which consists of two (disjoint) cycles of length one. To this end, we apply the functions in order; that is, we first apply Function \ref{Function:Split} to a portrait, then, when applying Functions \ref{Function:CriticalPointpre-periodToCycle} and \ref{Function:DecreasePeriod}, we may assume the critical points are in distinct components.  When applying Function \ref{Function:DecreasePeriod}, we may assume that the portrait has two cyclic components.

\begin{function}[Split portrait from one to two components]\label{Function:Split} Take as input a marked branched cover $f$ with marked set $M$ and associated portrait $\Gamma$ such that the portrait has one (non-empty) component.
\begin{enumerate}
    \item If the portrait of $f$ is a cycle of length $k$, that is: $C_1$ is periodic and $C_2$ is in the orbit of $C_1$ (and is hence also periodic), then there are two possibilities:
\begin{enumerate}
    \item If either $f(C_1) = C_2$ or $f(C_2)=C_1$, let $\phi$ be a half twist that swaps $C_1$ with $C_2$.  Output the marked branched cover $f\phi$, the marked set $M$, the portrait of $f\phi$, and the mapping class $\phi$.
    The portrait of $f \phi$ has two components, one where one critical point maps to itself, and the other critical point is in a cycle of period $k-1$.
    \item Otherwise, let $q$ be such that $f^q(C_1)=C_2$.  Let $\phi$ be a half twist that swaps $f^{q-1}(C_1)$ with $f^{k-1}(C_1)$ (where $k$ is the length of the cycle). Output the marked branched cover $f\phi$, the marked set $M$, the portrait of $f\phi$, and the mapping class $\phi$.  The portrait of $f\phi$ has two cyclic components -- $C_1$ is in a cycle of length $q$, $C_2$ is in a cycle of length $k-q$. 
\end{enumerate}

    \item Suppose the portrait of $f$ has one pre-period consisting of $\{C_1,f(C_1),\cdots, f^{r-1}(C_1)\}$, and $C_2$ is in the cycle of $C_1$.  Let $\diamond$ denote a pre-image of $C_1$.  Let $\phi$ be a half twist that swaps $f^{r-1}(C_1)$ with $\diamond$.  Output $f\phi$, the marked set $M\cup\{\diamond\}$, the portrait of $f\phi$ (in this case not all points of $M$ are vertices of the portrait), and the mapping class $\phi$.  Then the portrait of $f\phi$ has two cyclic components: $C_1$ is in a cycle of length $r$ and $C_2$ is in a cycle of length $k$ (where $k$ was the length of the cycle of the portrait of $f$).

    \item Suppose one critical point is in the pre-period of the other.  Let $C_2$ be in the orbit of $C_1$, and let $q$ be such that $f^q(C_1)=C_2$.  Let $\diamond$ denote a pre-image of $C_1$.   Let $\phi$ be a half twist that swaps $\diamond$ and $f^{q-1}(C_1)$ (which could be $C_1=f^0(C_1)$).  Output $f\phi$, the marked set $M\cup\{\diamond\}$, the portrait of $f\phi$ (in this case not all points of $M$ are vertices of the portrait), and the mapping class $\phi$.  Then the portrait of $f\phi$ has two components: $C_1$ is in a cyclic component of length $q-1$ and $C_2$ is in a component that is the same as the forward orbit of $C_2$ under $f$.  In particular, $C_1$ is cyclic under $f\phi$ and $C_2$ is strictly pre-periodic under $f\phi$.

    \item Suppose that the portrait of $f$ has two disjoint pre-periods. As above, let $r_i$ be the minimum $r_i$ such that $f^{r_i}(C_i)$ is periodic.  Let $\phi$ be a half twist that swaps $f^{r_1}(C_1)$ and $f^{r_2}(C_2)$.  Output $f\phi$, the marked set $M$, the portrait of $f\phi$, and the mapping class $\phi$.  Then the portrait of $f\phi$ has two pre-periodic components, one with a pre-period of length $r_1$ and a cycle of length $k_2$, and the other component with a pre-period of length $r_2$ and a cycle of length $k_1$.
    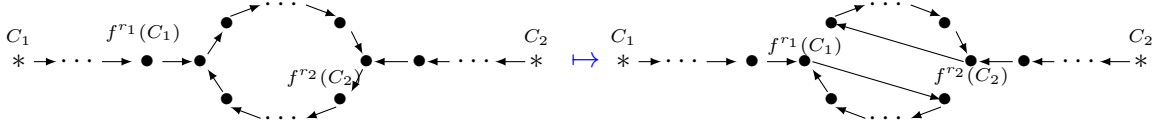
\begin{figure}
\begin{tikzpicture}
\node[label={[label distance=-4pt]90:\tiny $C_1$}] at (0,2) {$\ast$};
\node at (0.8,2) {$\cdots$};
\node at (1.7,2) {$\bullet$};
\node[label={[label distance=-4pt]135:{\tiny $f^{r_1}(C_1)$}}] at (2.4,2) {$\bullet$};
\node at (2.75,2.5) {$\bullet$};
\node at (3.5,2.75) {$\cdots$};
\node at (4.25,2.5) {$\bullet$};
\node[label={[label distance=-10pt]210:{\tiny$f^{r_2}(C_2)$}}] at (4.6,2) {$\bullet$};
\node at (4.25,1.5) {$\bullet$};
\node at (3.5,1.25) {$\cdots$};
\node at (2.75,1.5) {$\bullet$};
\node at (5.3,2) {$\bullet$};
\node at (6.05,2) {$\cdots$};
\node[label={[label distance=-4pt]90:\tiny $C_2$}] at (6.85,2) {$\ast$};
 \draw [-latex] (0.2,2) -- (0.5,2);
 \draw [-latex] (1.1,2) -- (1.5,2);
  \draw [-latex] (2.5,2.1) -- (2.7,2.4);
  \draw [-latex] (1.9,2) -- (2.3,2);
  \draw [-latex]  (2.8,2.6) -- (3.2,2.75);
  \draw [-latex]  (3.8,2.75) -- (4.2,2.6);
   \draw [-latex]  (4.4,2.4) -- (4.55,2.1);
   \draw [-latex]  (4.55,1.9) -- (4.4,1.6);
  \draw [-latex]  (4.2,1.4) -- (3.8,1.25);
  \draw [-latex]  (3.2,1.25) -- (2.8,1.4);
  \draw [-latex] (2.7,1.6) -- (2.5,1.9);
   \draw [-latex] (5.15,2) -- (4.7,2);
 \draw [-latex] (5.7,2) -- (5.4,2);
   \draw [-latex] (6.7,2) -- (6.3,2);

\node at (7.5,2) {\textcolor{blue}{$\mapsto$}};
     
\node[label={[label distance=-4pt]90:\tiny $C_1$}] at (8,2) {$\ast$};
\node at (8.8,2) {$\cdots$};
\node at (9.7,2) {$\bullet$};
\node[label={[label distance=-8pt]90:{\tiny $f^{r_1}(C_1)$}}] at (10.4,2) {$\bullet$};
\node at (10.75,2.5) {$\bullet$};
\node at (11.5,2.75) {$\cdots$};
\node at (12.25,2.5) {$\bullet$};
\node[label={[label distance=-8pt]-90:{\tiny $f^{r_2}(C_2)$}}] at (12.6,2) {$\bullet$};
\node at (12.25,1.5) {$\bullet$};
\node at (11.5,1.25) {$\cdots$};
\node at (10.75,1.5) {$\bullet$};
\node at (13.3,2) {$\bullet$};
\node at (14.05,2) {$\cdots$};
\node[label={[label distance=-4pt]90:\tiny $C_2$}] at (14.85,2) {$\ast$};
 \draw [-latex] (8.2,2) -- (8.5,2);
 \draw [-latex] (9.1,2) -- (9.5,2);
  \draw [-latex] (10.5,2) to (12.2,1.5);
  \draw [-latex] (9.9,2) -- (10.3,2);
  \draw [-latex]  (10.8,2.6) -- (11.2,2.75);
  \draw [-latex]  (11.8,2.75) -- (12.2,2.6);
   \draw [-latex]  (12.4,2.4) -- (12.55,2.1);
   \draw [-latex]  (12.5,2) -- (10.8,2.5);
  \draw [-latex]  (12.2,1.4) -- (11.8,1.25);
  \draw [-latex]  (11.2,1.25) -- (10.8,1.4);
  \draw [-latex] (10.7,1.6) -- (10.5,1.9);
   \draw [-latex] (13.15,2) -- (12.7,2);
 \draw [-latex] (13.7,2) -- (13.4,2);
   \draw [-latex] (14.7,2) -- (14.3,2);
\end{tikzpicture}

\caption{Function \ref{Function:Split}(4) may change all portrait features of $\Gamma_f$}\label{fig:split_disjoint}
\end{figure}

    \item Suppose that $C_1$ and $C_2$ are both strictly pre-periodic and their pre-periods are distinct and intersecting. Let $\ell_1>0,\ell_2\geq 0$ be such that $f^{\ell_1}(C_1) = f^{\ell_2}(C_2)$.  Let $\diamond$ denote a pre-image of the first critical point, $C_1$. Let $\phi$ be a half twist that swaps $\diamond$ with $f^{\ell_1-1}(C_1)$. Output $f\phi$, the marked set $M\cup\{\diamond\}$, the portrait of $f\phi$ (in this case not all points of $M$ are vertices of the portrait), and the mapping class $\phi$.  The portrait $f \phi$ has two components: $C_1$ is in a cyclic component of length $\ell_1$ and $C_2$ is in a component that is the same as the forward orbit of $C_2$ under $f$. 
\end{enumerate}
Function \ref{Function:Split} can change all portrait features of $f$.
\begin{figure}
\begin{tikzpicture}
\node[label={\footnotesize $C_1$}] at (1,3.75) {$\ast$};
\node at (1,3) {$\diamond$};
\node at (1.6,3.35) {$\bullet$};
\node at (2.4,2.95) {$\ddots$};
\node [label={[label distance=-8pt]20:$f^{s-1}(C_1)$}] at (3.2,2.4) {$\bullet$};
\node[label={\footnotesize $C_2$}] at (1,.25) {$\ast$};
\node at (1.6,.6) {$\bullet$};
\node at (2.4,1.2) {\reflectbox{$\ddots$}};
\node[label={[label distance=-10pt]340:$f^{r-1}(C_2)$}] at (3.2,1.6) {$\bullet$};
\node at (3.8,2) {$\bullet$};
\node at (4.7,2) {$\cdots$};
\node at (5.4,2) {$\bullet$};
\node at (5.75,2.5) {$\bullet$};
\node at (6.5,2.75) {$\cdots$};
\node at (7.25,2.5) {$\bullet$};
\node at (7.6,2) {$\bullet$};
\node at (7.25,1.5) {$\bullet$};
\node at (6.5,1.25) {$\cdots$};
\node at (5.75,1.5) {$\bullet$};

\draw[-latex] (1,3.15) -- (1,3.65);
\draw [-latex] (1.1,3.7) -- (1.45,3.45);
\draw [-latex] (1.7,3.3) -- (2.1,3);
 \draw [-latex] (2.7,2.7) -- (3.1,2.45);
 \draw [-latex] (3.3,2.35) -- (3.75,2.05);
  \draw [-latex] (2.7,1.3) -- (3.1,1.55);
  \draw [-latex] (3.3,1.65) -- (3.75,1.95);
  \draw [-latex] (1.1,.3) -- (1.45,.55);
\draw [-latex] (1.7,.68) -- (2.1,.94);
 \draw [-latex] (3.95,2) -- (4.4,2);
  \draw [-latex] (5.55,2.1) -- (5.7,2.4);
  \draw [-latex] (4.9,2) -- (5.3,2);
  \draw [-latex]  (5.8,2.6) -- (6.2,2.75);
  \draw [-latex]  (6.8,2.75) -- (7.2,2.6);
   \draw [-latex]  (7.4,2.4) -- (7.55,2.1);
   \draw [-latex]  (7.55,1.9) -- (7.4,1.6);
  \draw [-latex]  (7.2,1.4) -- (6.8,1.25);
  \draw [-latex]  (6.2,1.25) -- (5.8,1.4);
  \draw [-latex] (5.7,1.6) -- (5.5,1.9);
  \draw [ultra thick, blue, latex'-latex'] (1.12,3) -- (3.1,2.4);

\node at (8.5,2) {\textcolor{blue}{$\mapsto$}};
     
\node[label={\footnotesize $C_1$}] at (9,3.75) {$\ast$};
\node at (9,3) {$\diamond$};
\node at (9.6,3.35) {$\bullet$};
\node at (10.4,2.95) {$\ddots$};
\node[label={[label distance=-8pt]20:$f^{s-1}(C_1)$}] at (11.2,2.4) {$\bullet$};
\node[label={\footnotesize $C_2$}] at (9,.25) {$\ast$};
\node at (9.6,.6) {$\bullet$};
\node at (10.4,1.2) {\reflectbox{$\ddots$}};
\node[label={[label distance=-10pt]340:$f^{r-1}(C_2)$}] at (11.2,1.6) {$\bullet$};
\node at (11.8,2) {$\bullet$};
\node at (12.7,2) {$\cdots$};
\node at (13.4,2) {$\bullet$};
\node at (13.75,2.5) {$\bullet$};
\node at (14.5,2.75) {$\cdots$};
\node at (15.25,2.5) {$\bullet$};
\node at (15.6,2) {$\bullet$};
\node at (15.25,1.5) {$\bullet$};
\node at (14.5,1.25) {$\cdots$};
\node at (13.75,1.5) {$\bullet$};

\draw [-latex] (9.1,3.7) -- (9.45,3.45);
\draw [-latex] (9.7,3.3) -- (10.1,3);
 \draw [-latex] (10.7,2.7) -- (11.1,2.45);
 \draw [-latex] (11.25,2.4) to [out=90,in=30] (9.1,3.75);
  \draw [-latex] (10.7,1.3) -- (11.1,1.55);
  \draw [-latex] (11.3,1.65) -- (11.75,1.95);
  \draw [-latex] (9.1,.3) -- (9.45,.55);
\draw [-latex] (9.7,.68) -- (10.1,.94);
 \draw [-latex] (11.95,2) -- (12.4,2);
  \draw [-latex] (13.55,2.1) -- (13.7,2.4);
  \draw [-latex] (12.9,2) -- (13.3,2);
  \draw [-latex]  (13.8,2.6) -- (14.2,2.75);
  \draw [-latex]  (14.8,2.75) -- (15.2,2.6);
   \draw [-latex]  (15.4,2.4) -- (15.55,2.1);
   \draw [-latex]  (15.55,1.9) -- (15.4,1.6);
  \draw [-latex]  (15.2,1.4) -- (14.8,1.25);
  \draw [-latex]  (14.2,1.25) -- (13.8,1.4);
  \draw [-latex] (13.7,1.6) -- (13.5,1.9);
\end{tikzpicture}
  \caption{Function \ref{Function:Split}(5) may change all portrait features of $\Gamma_f$.}\label{fig:Split5}
  \end{figure}
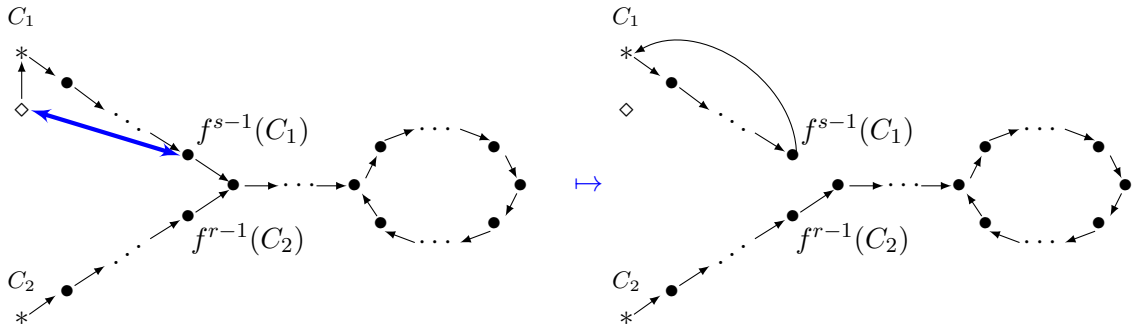
\end{function}

\begin{proof}
(1) In (a), assume up to relabeling that $f(C_1)=C_2$.  As defined in the statement of the Function, $\phi$ swaps $C_1$ and $C_2$.  Then $f\phi$ maps $C_1$ to itself and maps $C_2$ to $f(C_2)=f^2(C_1)$.  Since $\phi$ can be chosen to be supported on a neighborhood of an arc between $C_1$ and $C_2$, $(f\phi)^\ell(C_1)=f^\ell(C_2)$ for all $2\leq \ell\leq k-1$.  In particular, the portrait of $f\phi$ has two components: one consisting of the fixed point $C_1$, the other is a cyclic component consisting of $\{C_2,f^2(C_1),\cdots,f^{k-1}(C_1)\}$.

In (b), then $\phi$ swaps $f^{q-1}(C_1)$ with $f^{k-1}(C_1)$.  Then $f\phi$ maps $f^{q-1}(C_1)$ to $C_1$ and maps $f^{k-1}(C_1)$ to $C_2$. Since $\phi$ can be chosen to be supported on a neighborhood of an arc between $f^{q-1}(C_1)$ and $f^{k-1}(C_1)$, $f\phi$ and $f$ are equal for all other post-critical points.  In particular, the portrait of $f\phi$ has two cyclic components, one that contains $\{C_1,f(C_1),\cdots,f^{q-1}(C_1)\}$ and one that contains $\{C_2=f^q(C_1),f^{q+1}(C_1),\cdots,f^{k-1}(C_1)\}$.

\medskip
(2) The branched cover $f\phi$ maps $f^{r-1}(C_1)$ to $C_1$.
Since $\phi$ can be chosen so that it is supported on a neighborhood of an arc containing $\diamond$ and $f^{r-1}(C_1)$, the maps $f$ and $f\phi$ agree on all critical and post-critical points except for $f^{r-1}(C_1)$.

\medskip
(3) The branched cover $f \phi$ maps $f^{q-1}(C_1)$ to $C_1$.
Since $\phi$ can be chosen so that it is supported on a neighborhood of an arc containing $\diamond$ and $f^{q-1}(C_1)$, the maps $f$ and $f\phi$ agree on all critical and post-critical points except for $f^{q-1}(C_1)$.  Then $C_1$ is periodic under $f\phi$ with period consisting of $\{C_1,f(C_1),\cdots,f^{q-1}(C_1)\}$.  The other component is pre-periodic and consists of $\{C_2,f^{q+1}(C_1),\cdots,f^{r+k-1}(C_1)\}$ where the pre-period is the intersection of the pre-periods of $C_1$ and $C_2$ under $f$ (which is exactly the pre-period of $C_2$ under $f$).

\medskip
(4) The branched cover $f\phi$ maps $f^{r_1}(C_1)$ to $f^{r_2+1}(C_2)$ and maps $f^{r_2}(C_2)$ to $f^{r_1+1}(C_1)$. 
Since $\phi$ can be chosen to be supported on the neighborhood of an arc between $f^{r_1}(C_1)$ and $f^{r_2}(C_2)$, $f\phi$ and $f$ are equal for all critical and post-critical points except $f^{r_1}(C_1)$ and $f^{r_2}(C_2)$. The portrait of $f\phi$ has two components: one has pre-period $\{C_1,f(C_1),\cdots,f^{r_1-1}(C_1)\}$ and cycle $\{f^{r_1}(C_1),f^{r_2+1}(C_2),\cdots, f^{r_2+k_2-1}(C_2)\}$.  The other has pre-period $\{C_2,f(C_2),\cdots,f^{r_2-1}(C_2)\}$ and cycle $\{f^{r_2}(C_2),f^{r_1+1}(C_1),\cdots, f^{r_1+k_1-1}(C_1)\}$.

\medskip
(5) The branched cover $f\phi$ maps $f^{\ell_1-1}(C_1)$ to $C_1$. 
Since $\phi$ can be chosen so that it is supported on a neighborhood of an arc containing $\diamond$ and $f^{\ell_1-1}(C_1)$, $f$ and $f\phi$ are equal on all critical and post-critical points except for $f^{\ell_1-1}(C_1)$. 
Then $C_1$ is periodic under $f\phi$ with period $\{C_1,f(C_1),\cdots,f^{\ell_1-1}(C_1)\}$. 
The other component of the portrait of $f\phi$ consists of the orbit of $C_2$ and the pre-period is the same as the pre-period of $C_2$ under $f$ and the cycle is the same as the cycle of $C_2$ under $f$.
\end{proof}

\begin{function}[Make a strictly pre-periodic critical point periodic]\label{Function:CriticalPointpre-periodToCycle}  Take as input a marked branched cover $f$ with marked set $M$ and associated portrait $\Gamma$ with two non-empty components such that (at least) one critical point is strictly pre-periodic. 
 Suppose the critical point $C_j$ is strictly pre-periodic under $f$.  Let $r$ be the length of the pre-period of $C_j$ under $f$, and let $k$ be the length of the cycle of $C_j$ under $f$.  Because $C_1$ and $C_2$ are in different components, neither critical point is the image of the other.  Therefore $f^{-1}(C_j)$ consists of two points, and at most one of them is post-critical.  Let $\diamond$ be a preimage of $C_j$ that does not lie in the post-critical set of $f$.  Let $\phi$ be a half twist that swaps $\diamond$ and $f^{r+k-1}(C_j)$ (the last point of the cycle of $C_j$ under $f$).  Output the marked branched cover $f\phi$, the marked set $M\cup\{\diamond\}$, the portrait of $f\phi$, and the mapping class $\phi$.  Then $C_j$ is in the cycle of $f\phi$ and the cycle has length $k+r$.   We refer to the change in the portrait that results from applying $\phi$ as ``applying Function \ref{Function:CriticalPointpre-periodToCycle} to $C_j$."
   The branched cover $f \phi$ does not change the number of portrait components of $\Gamma_f$.
\end{function}

\begin{proof}
The branched cover $f\phi$ maps $f^{r+k-1}(C_j)$ to $C_j$. 
Since $\phi$ can be chosen so that it is supported on a neighborhood of an arc between $\diamond$ and $f^{r+k-1}(C_j)$, the maps $f$ and $f\phi$ agree on all critical and post-critical points except $f^{r+k-1}(C_j)$.  
Then the cycle of $C_j$ under $f\phi$ is $\{f(C_j),\cdots,f^{r-1}(C_j),f^r(C_j),\cdots,f^{r+k-1}(C_j),C_j\}$ (where $\{f(C_j),\cdots,f^{r-1}(C_j)\}$ is the pre-period of $C_j$ under $f$ and $\{f^r(C_j),\cdots,f^{r+k-1}(C_j)\}$ is the cycle of $C_j$ under $f$).
In particular, the portrait of $f\phi$ and the portrait of $f$ have the same number of components.
\end{proof}

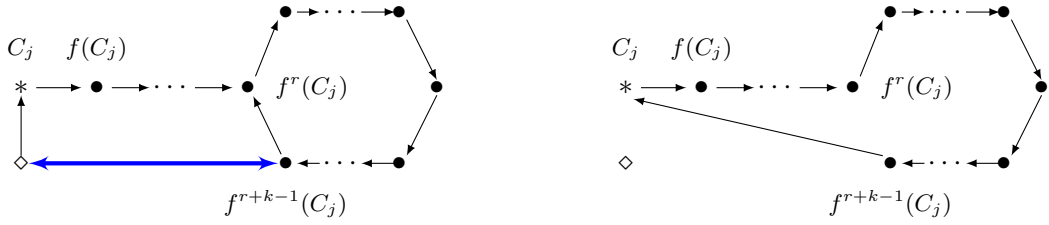
\begin{figure}
\begin{tikzpicture}\label{fig:pre-period_to_cycle}
\node[label={\footnotesize $C_j$}] at (1,2) {$\ast$};
\node at (1,1) {$\diamond$};
\node[label={\footnotesize $f(C_j)$}] at (2,2) {$\bullet$};
\node at (3,2) {$\cdots$};
\node[label=right:{\footnotesize $f^{r}(C_j)$}] at (4,2) {$\bullet$};
\node at (4.5,3) {$\bullet$};
\node at (5.25,3) {$\cdots$};
\node at (6,3) {$\bullet$};
\node at (6.5,2) {$\bullet$};
\node at (6,1) { $\bullet$};
\node at (5.25,1) {$\cdots$};
\node[label=below:{\footnotesize $f^{r+k-1}(C_j)$}] at (4.5,1) {$\bullet$};
\draw [-latex] (1,1.1) -- (1,1.9);
 \draw [-latex] (1.2,2) -- (1.8,2);
 \draw [-latex] (2.2,2) -- (2.7,2);
  \draw [-latex] (3.3,2) -- (3.8,2);
  \draw [-latex] (4.1,2.15) -- (4.4,2.9);
    \draw [-latex]  (4.65,3) -- (4.95,3);
  \draw [-latex]  (5.5,3) -- (5.95,3);
  \draw [-latex]  (6.1,2.9) -- (6.5,2.1);
    \draw [-latex]  (6.5,1.9) -- (6.1,1.1);
        \draw [-latex]  (4.44,1.12) -- (4.07,1.89);
    \draw [-latex]  (4.95,1) -- (4.65,1);
  \draw [-latex]  (5.95,1) -- (5.5,1);
  \draw[ultra thick, blue, latex'-latex'] (1.1,1) -- (4.4,1);

\node[label={\footnotesize $C_j$}] at (9,2) {$\ast$};
\node at (9,1) {$\diamond$};
\node[label={\footnotesize $f(C_j)$}] at (10,2) {$\bullet$};
\node at (11,2) {$\cdots$};
\node[label=right:{\footnotesize $f^{r}(C_j)$}] at (12,2) {$\bullet$};
\node at (12.5,3) {$\bullet$};
\node at (13.25,3) {$\cdots$};
\node at (14,3) {$\bullet$};
\node at (14.5,2) {$\bullet$};
\node at (14,1) { $\bullet$};
\node at (13.25,1) {$\cdots$};
\node[label=below:{\footnotesize $f^{r+k-1}(C_j)$}] at (12.5,1) {$\bullet$};
 \draw [-latex] (9.2,2) -- (9.8,2);
 \draw [-latex] (10.2,2) -- (10.7,2);
  \draw [-latex] (11.3,2) -- (11.8,2);
  \draw [-latex] (12.1,2.15) -- (12.4,2.9);
    \draw [-latex]  (12.65,3) -- (12.95,3);
  \draw [-latex]  (13.5,3) -- (13.95,3);
  \draw [-latex]  (14.1,2.9) -- (14.5,2.1);
    \draw [-latex]  (14.5,1.9) -- (14.1,1.1);
        \draw [-latex]  (12.4,1.1) -- (9.1,1.85);
    \draw [-latex]  (12.95,1) -- (12.65,1);
  \draw [-latex]  (13.95,1) -- (13.5,1);
\end{tikzpicture}
  \caption{Applying Function \ref{Function:CriticalPointpre-periodToCycle} to the critical point $C_j$ creates a new function in for which $C_j$ periodic.}
  \end{figure}

\begin{function}[Decrease a cycle length]\label{Function:DecreasePeriod}
Take as input a marked branched cover $f$ with marked set $M$ and portrait $\Gamma$.  We will assume that the portrait of $f$ consists of two cyclic components.  If the length of the cycle of $C_j$ is greater than one, let $f^\ell(C_j)$ and $f^{\ell+1}(C_j)$ be (distinct) post-critical points in the cycle of $C_j$.  The points $f^{\ell}(C_j)$ and $f^{\ell+1}(C_j)$ will not both be critical points because each critical point is in a different component.
    Then let $\phi$ be a half twist that swaps $f^\ell(C_j)$ and $f^{\ell+1}(C_j)$.  Output the marked branched cover $f\phi$, the marked set $M$, the portrait of $f\phi$, and the mapping class $\phi$.
    The cycle of $C_j$ in the portrait of $f\phi$ has length one less than the the length of the cycle of $C_j$ in the portrait of $f$.

The branched cover $f\phi$ only affects the length of the cycle, and no other portrait features of $\Gamma_f$.
 We refer to the change in the portrait that results from applying $\phi$ as ``applying Function \ref{Function:DecreasePeriod} to the point $f^\ell(C_j)$."
\end{function}
\begin{proof}
Precomposing  with a half twist which swaps $f^\ell(C_j)$ and $f^{\ell+1}(C_j)$ results in $f^{\ell}(C_j)$ mapping to $f^{\ell+2}(C_j)$ and $f^{\ell+1}(C_j)$ mapping to itself. Therefore the portrait for $f\phi$ has one less point than $\Gamma_f$.  
 Because $\phi$ is supported on a neighborhood of an arc between $f^\ell(C_j)$ and $f^{\ell+1}(C_1)$, no other cycles or pre-periods of $\Gamma_f$ change under the map $\phi$.  That is: the length of the cycle decreases by one and all other portrait features remain unchanged.
\end{proof}

\begin{proposition}
    The set of abstract portraits that are compatible with a quadratic branched cover are closed under Functions \ref{Function:Split}-\ref{Function:DecreasePeriod}.
\end{proposition}
\begin{proof}
   Let $f$ be an abstract portrait compatible with a quadratic branched cover and let $\phi$ be a homeomorphism used in one of Functions \ref{Function:Split}-\ref{Function:DecreasePeriod}.  We claim that $f\phi$ is compatible with a quadratic branched cover.  First we note that conditions (1)-(5) in the definition of abstract portrait are used to determine which points in the marked set of $f\phi$ are in the portrait $\Gamma_{f\phi}$ and which points are distinguished as critical points in $\Gamma_{f\phi}$. Because these serve to define the points, they do not need to be verified. 
   
   There are three properties required for $\Gamma_{f\phi}$ to be compatible with a quadratic branched cover: it has exactly two points distinguished as critical points, the edges originating at the points distinguished as critical points are labeled with a 2, and the sum of the labels of edges that terminate at each vertex are at most 2.  The first property follows from the fact that $\phi$ is a homeomorphism.  The second property holds because none of the functions change the labels of edges: edges originating at points distinguished as critical points will be labeled with a 2 and all other edges will be labeled with a 1 (even if the points distinguished as critical points are different).  To ensure the third property is satisfied, we must check that for any point that is the terminus of two edges $e_1$ and $e_2$, that no function swaps the origin of $e_1$ or $e_2$ with a critical point.

   Function \ref{Function:Split} has six cases, but all cases either swap two critical points, two non-critical points, or swap a point that is not a vertex of $\Gamma$ with a point that is a vertex of $\Gamma$.  Therefore no cases of Function \ref{Function:Split} swap a critical point with a non-critical origin of an edge in $\Gamma$.  Function \ref{Function:CriticalPointpre-periodToCycle} adds a point that is not in $\Gamma$ to the portrait, and therefore do not swap a critical point with the origin of any edge in $\Gamma$.  The homeomorphism in Function \ref{Function:DecreasePeriod} swaps two post-critical points, but is only applied under the assumption that the components are both cycles and therefore has no points that are the terminus of two edges.  
\end{proof}

\section{Algorithm}\label{sec:algorithm}
Let $g,h:S^2\rightarrow S^2$ be a post-critically finite branched covers with portrait $\Gamma_g$ and $\Gamma_h$, respectively.  Let $\Gamma$ denote the portrait of $p(z)=z^2$, which consists of the two fixed critical points 0 and $\infty$.  The algorithm below describes a sequence of portraits $\Gamma^0,\cdots,\Gamma^n$ with $\Gamma_g=\Gamma^0$ and $\Gamma=\Gamma^n$.  For each step of the algorithm, we compare the portrait $\Gamma^i$ to the portrait $\Gamma$ and apply one of the functions from Section \ref{sec:functions} to obtain the next portrait $\Gamma^{i+1}$.  Then we repeat the algorithm for $h$.  The algorithm tracks the marked set in each step and the mapping class that is composed in each step.  Note that in each step, the domain is a sphere, but the marked set may be a proper superset of the marked set in the previous step.  We are able to compose mapping classes from each step, possibly by applying capping homomorphisms to mapping classes from some steps \cite{FM1}.  The final marked set $M$ is the union of the marked set in each step of the algorithm for both $g$ and $h$.
The final mapping class $\phi$ is the composition of the mapping classes from each step of the algorithm for $g$ and the inverse of the mapping class from each step of the algorithm for $h$.

\p{Components of $\Gamma_g$, $\Gamma_h$, and $\Gamma$} Because a portrait of a quadratic branched cover can have at most two (non-empty) components, we will consider each portrait as a pair of (possibly empty) components $\Gamma_g=(\Gamma^g_1,\Gamma^g_2)$, $\Gamma_h=(\Gamma^h_1,\Gamma^h_2)$, and $\Gamma=(\Gamma_1,\Gamma_2)$.  If one component is empty, we will always label the empty component as the second entry.  Moreover, we will say that a portrait has {\it one component} if it has one non-empty component and one empty component.   

\p{The algorithm} Steps \ref{step:number}- \ref{step:cycle_length} below give a sequence of portraits $\Gamma^i=\{(\Gamma^i_1,\Gamma^i_2)\}$ so that $\Gamma^{i+1}_j$ differs from $\Gamma^i_j$ by a function from Section \ref{sec:functions} 
and there exists $n\geq 1$, so that $\Gamma^n=\Gamma$.

We first initialize the algorithm with the pair $(\Gamma^0,f_0)=(\Gamma_g,g)$, the set $M_0$ consisting of $P_g$ and the critical points of $g$, the identity mapping class $\phi$, and the sequence (at this stage consisting of the single entry) $\{(\Gamma^0,f_0)\}$.  After completing the algorithm for $g$, we initialize the algorithm with the pair $(\Gamma^0,f_0)=(\Gamma_h,h)$, the set $M_0$ consisting of $P_h$ and the critical points of $h$, the identity mapping class $\phi$, and the (single entry) sequence $\{(\Gamma_0,f_0)\}$.

\medskip

\begin{step}[Number of components]\steplabel{step:number}
Take as input the pair $(\Gamma^i,f_i)$, the marked set $M_i$, a mapping class $\phi_i$, and the sequence $\{(\Gamma^0,f_0),(\Gamma^1,f_1),\cdots,(\Gamma^i,f_i)\}$.  As above, each $\Gamma^i$ is a pair of (possibly empty) components.

 If $\Gamma^i$ has two non-empty components, return $(\Gamma^i,f_i)$, the marked set $M_i$, the mapping class $\phi_i$, and the sequence $\{(\Gamma^0,f_0),(\Gamma^1,f_1),\cdots,(\Gamma^i,f_i)\}$, and continue to Step \ref{step:pre-period_structure2}.

If $\Gamma^i$ has one component, then we apply Function \ref{Function:Split}.
In particular, if the component of $\Gamma^i$ is a cycle, apply Function \ref{Function:Split} (1).
If the component of $\Gamma^i$ has one pre-period, apply Function \ref{Function:Split} (2).  If $\Gamma^i$ has two pre-periods where one is contained in the other, apply Function \ref{Function:Split} (3).
If $\Gamma^i$ is a portrait which has two disjoint pre-periods, apply Function \ref{Function:Split} (4).
If $\Gamma^i$ is a portrait which has two pre-periods which are distincting and intersecting, apply Function \ref{Function:Split} (5).
Let $(f_{i+1},M_{i+1},\Gamma_{i+1},\phi_{i+1})$ be the output of Function \ref{Function:Split}.  Return $(f_{i+1},M_{i+1},\Gamma_{i+1},\phi_{i}\circ\phi_{i+1},\{ (\Gamma^0,f_0), \cdots, (\Gamma^{i+1},f_{i+1})\})$.
Continue to Step \ref{step:pre-period_structure2}. 
\end{step}

\begin{step}[Pre-period structure]\steplabel{step:pre-period_structure2}
Take as input the pair $(\Gamma^i,f_i)$, the marked set $M_i$, a mapping class $\phi_i$, and the sequence $\{(\Gamma^0,f_0),(\Gamma^1,f_1),\cdots,(\Gamma^i,f_i)\}$.  
Let $\Gamma^i=(\Gamma_1^i,\Gamma_2^i)$ and $\Gamma=(\Gamma_1,\Gamma_2)$.  
We may assume that $\Gamma^i$ has two non-empty components because Step \ref{step:pre-period_structure2} is only called in this case.  
Let $C_1^i$ and $C_2^i$ be the critical points of $f_i$ and, as above.

If both components of $\Gamma^i$ are cycles, continue to Step \ref{step:cycle_length}.

 If $C_1^i$ is strictly pre-periodic under $f_i$, then apply Function \ref{Function:CriticalPointpre-periodToCycle} to $C_1^i$.  
Let $(f_{i+1},M_{i+1},\Gamma_{i+1},\phi_{i+1})$ be the output of Function \ref{Function:CriticalPointpre-periodToCycle}.  Return $(f_{i+1},M_{i+1},\Gamma_{i+1},\phi_{i}\circ\phi_{i+1},\{ (\Gamma^0,f_0), \cdots, (\Gamma^{i+1},f_{i+1})\})$.  Continue to Step \ref{step:pre-period_structure2}.

 If $C_2^i$ is strictly pre-periodic under $f_i$, then apply Function \ref{Function:CriticalPointpre-periodToCycle} to $C_2^i$. Let $(f_{i+1},M_{i+1},\Gamma_{i+1},\phi_{i+1})$ be the output of Function \ref{Function:CriticalPointpre-periodToCycle}.  Return $(f_{i+1},M_{i+1},\Gamma_{i+1},\phi_{i}\circ\phi_{i+1},\{ (\Gamma^0,f_0), \cdots, (\Gamma^{i+1},f_{i+1})\})$. Continue to Step \ref{step:pre-period_structure2}.
\end{step}

\begin{step}[Cycle lengths]\steplabel{step:cycle_length}
Take as input the pair $(\Gamma^i,f_i)$, the marked set $M_i$, a mapping class $\phi_i$, and the sequence $\{(\Gamma^0,f_0),(\Gamma^1,f_1),\cdots,(\Gamma^i,f_i)\}$.  We may assume that $\Gamma^i$ and $\Gamma$ have two cyclic components (by Steps \ref{step:number} and \ref{step:pre-period_structure2}).

Since $\Gamma^i$ has two components, $C_1^i$ and $C_2^i$ lie in separate portrait components. Let $k^i_j$ be the length of the cycle of the component $\Gamma^i_j$ of $\Gamma^i$.  

    If $k^i_1 = 1$ and $k^i_2=1$, then the algorithm terminates.  Return $(\Gamma^i, f_i)$, the marked set $M_i$, the mapping class $\phi_i$, and the sequence $\{ (\Gamma^0,f_0), (\Gamma^1,f_1), \cdots, (\Gamma^i,f_i) \}$.
    
\noindent Otherwise, at least one of $k^i_1 >1$ or $k^i_2 >1$.

\hspace{\parindent}If $k^i_1 > 1$ apply Function \ref{Function:DecreasePeriod} to a non-critical point in the cycle of $\Gamma_1^i$.  Let $(f_{i+1},M_{i+1},\Gamma_{i+1},\phi_{i+1})$ be the output of Function \ref{Function:DecreasePeriod}(1).  Return $(f_{i+1},M_{i+1},\Gamma_{i+1},\phi_{i}\circ\phi_{i+1},\{ (\Gamma^0,f_0), \cdots, (\Gamma^{i+1},f_{i+1})\})$.  Repeat Step \ref{step:cycle_length}.

If $k^i_2 > 1$ apply Function \ref{Function:DecreasePeriod} to a non-critical point in the cycle of $\Gamma_2^i$. Let $(f_{i+1},M_{i+1},\Gamma_{i+1},\phi_{i+1})$ be the output of Function \ref{Function:DecreasePeriod}(1).  Return $(f_{i+1},M_{i+1},\Gamma_{i+1},\phi_{i}\circ\phi_{i+1},\{ (\Gamma^0,f_0), \cdots, (\Gamma^{i+1},f_{i+1})\})$.  Repeat Step \ref{step:cycle_length}.

\end{step}

\section{Proof of Theorem \ref{thm:main}}\label{sec:proof}
We apply the algorithm in Section \ref{sec:algorithm} to the functions $g$ and $h$.  Each application returns a 5-tuple $(\Gamma_n,f_n,M_n,\phi_n,\{(\Gamma^0,f_0),(\Gamma^1,f_1),\cdots,(\Gamma^n,f_n)\})$, where $f_n$ is equivalent to $z^2$ and $\Gamma^n$ is the portrait for $z^2$.  Let $M$ be the union of the sets $M_n$ obtained when applying the algorithm to $g$ and to $h$.  Let $\phi_n^g$ be the mapping class from applying the algorithm to $g$ and let $\phi_n^h$ be the mapping class from applying the algorithm to $h$.  The mapping class $\phi_n^g\circ(\phi_n^h)^{-1}\in\Mod(S^2,M)$ in the desired mapping class in Theorem \ref{thm:main}.  Before we prove the Theorem \ref{thm:main}, we need to prove that the algorithm terminates in a finite number of steps (hence proving that $M$ is finite).
\begin{lemma}\label{Lem:EachStepFinite}
Let $g:S^2\rightarrow S^2$ be a marked quadratic branched cover with dynamical portrait $\Gamma_g$.  Let $\Gamma$ be the portrait for $p(z)=z^2$.  The algorithm that runs by calling Step \ref{step:number} (which in turn calls either Step \ref{step:pre-period_structure2}  and \ref{step:cycle_length}) terminates in a finite number of steps.
\end{lemma}

\begin{proof}
    For each step, we calculate the maximum number of times the step runs in order to obtain an upper bound on the steps required for the algorithm to run.

    \textbf{Step \ref{step:number}} \textit{Runs twice:} once for $g$ and once for $h$.

    \textbf{Step \ref{step:pre-period_structure2}} \textit{runs at most six times}.  The step runs at most once for each component of both $g$ and $h$, plus a final time (for each $g$ and $h$) to verify that both components are cycles: a total of six times.

    \textbf{Step \ref{step:cycle_length}:}
We consider the number of times Step \ref{step:cycle_length} is run for both $g$ and $h$.
 Let $(\Gamma^i,f_i)$ be the pair returned from Step \ref{step:pre-period_structure2} when the algorithm is run for $g$ and let $(\Delta^i,f_i)$ be the pair returned from step \ref{step:pre-period_structure2} when the algorithm is run for $h$.
    Let $k_j^g$ be the cycle length of $\Gamma^i_j$ (when the algorithm is run for $g$) and $k_j^h$ be the cycle length of $\Delta_i^j$ (when the algorithm is run for $h$).  Then Step \ref{step:cycle_length} runs $k_1^g+k_2^g-1$ for $g$.  Indeed, Step \ref{step:cycle_length} runs $k_1^g-1$ times for the first component, $k_2^g-1$ times for the second component and one additional time to check that the lengths are correct.  A similar calculation shows that Step \ref{step:cycle_length} runs $k_1^h+k_2^h-1$ for $h$.  Hence Step \ref{step:cycle_length} runs a total of $k_1^g+k_2^g+k_1^h+k_2^h-2$ times.

    \textbf{Conclusion:} Therefore, we see that the algorithm runs each step a finite number of times, which is what we aimed to show. 
\end{proof}

The final ingredient is a more general restatement of \cite[Lemma 3.6]{KL}.  Let $Q(n)$ be the set of quadratic branched covers with $n$ post-critical points.  The proof is identical to that of Kelsey--Lodge, by replacing the set $Q(4)^*$ with $Q(n)$.
\begin{lemma}\label{lem:uniqueness}
    Let $f,g\in Q(n)$.  Then $g,h$ are pure Hurwitz equivalent if and only if they have equivalent portraits.
\end{lemma}

Finally, we are able to prove Theorem \ref{thm:main}.

\begin{proof}[Proof of Theorem \ref{thm:main}]
Let $g$ and $h$ be the branched covers indicated in the theorem.  Let $\Gamma_g$ be the portrait of $g$ and let $M_g$ be the marked set of $g$. 
 Let $\Gamma$ be the portrait of $z^2$.  We first run the algorithm by inputting $(\Gamma_g,g)$, the marked set $M_g$, the identity mapping class, and the sequence $\{(\Gamma_g=\Gamma_0,g=f_0)\}$ into Step \ref{step:number}.  Let $(\Gamma_n,f_n,M_n^f,\phi_n^f,\{(\Gamma^0,f_0),(\Gamma^1,f_1),\cdots,(\Gamma^n,f_n)\})$ be the return of Step \ref{step:cycle_length} when the cycle terminates.  By construction $\phi_i^g\in\Mod(S^2,M_i^g)$ for all $0\leq i\leq n$.  In particular, $\phi_n^g\in\Mod(S^2,M_n^g)$. Next we prove that $g\circ\phi^f_n$ has the same portrait as $z^2$.

 First we observe that $\Gamma^{1}$ and $\Gamma$ both have two components by applying  Step \ref{step:number}.  Steps \ref{step:pre-period_structure2} and \ref{step:cycle_length} call Functions\ref{Function:CriticalPointpre-periodToCycle} and \ref{Function:DecreasePeriod}, neither of which changes the number of components of $\Gamma^i$. 
 
Next, let $i_2$ be the index of the final iteration of Step \ref{step:pre-period_structure2} (that calls Step \ref{step:cycle_length}).  Observe that $\Gamma^{i_2}$ has two components that are each cycles, just like $\Gamma$.  The final iteration of Step \ref{step:pre-period_structure2} verifies that $\Gamma^{i_2}$ and $\Gamma$ have the same pre-period structure before moving to Step \ref{step:cycle_length}.  Step \ref{step:cycle_length} calls Function \ref{Function:DecreasePeriod}, which does not change pre-period structure of $\Gamma^i$.

 The final iteration of Step \ref{step:cycle_length} ensures that  $\Gamma^{n}_f$ consists of two components, both of which are cycles of length 1.  This is the same portrait as $\Gamma$.

 We repeat the process for $h$ by inputting $(\Gamma_h,h)$, the marked set $M_h$, the identity mapping class and the sequence $\{(\Gamma_h=\Gamma_0,h=f_0)\}$ into Step \ref{step:number}.  The final iteration of Step \ref{step:cycle_length} will output a sequence $(\Gamma_m,f_m,M_m^h,\phi_m^h,\{(\Gamma^0,f_0),(\Gamma^1,f_1),\cdots,(\Gamma^m,f_m)\})$.  Similarly to $g$, we we obtain the mapping class $\phi_m^h\in\Mod(S^2,M_m^h)$.  The final portrait is the same as $\Gamma$.

 Then $g\circ\phi_n^g$ and $h\circ\phi_m^h$ both have portraits isomorphic to $\Gamma$.  Let $\phi=\phi_n^g\circ(\phi_m^h)^{-1}\in\Mod(S^2, M_n^g\cup M_m^h)$.  Then $g\circ\phi=(g\circ \phi_n^g)\circ(\phi_m^h)^{-1}$ has portrait isomorphic to the portrait of $h$.

 Finally, because the portraits of $g\circ\phi$ and $h$ are equivalent, by Lemma \ref{lem:uniqueness} (cf. \cite[Lemma 3.6]{KL}), $g\circ\phi$ and $h$ are pure Hurwitz equivalent.
\end{proof}

\bibliographystyle{plain}
\bibliography{transitivity}
\end{document}